\documentclass[a4paper,oneside,portrait,12pt]{AmsArt}

\usepackage[margin=1in]{geometry}

\usepackage {graphicx}
\usepackage{amsfonts}
\usepackage{amsthm}
\usepackage{amsmath}
\usepackage{amsfonts}
\usepackage{latexsym}
\usepackage{amssymb}
\usepackage{epsfig,color}
\usepackage{graphicx, amssymb}

\usepackage{hyperref}
\hypersetup{
  bookmarks=true,
}

\newcommand{\A}{\mathcal{A}_{reg}}
\newcommand{\M}{\mathcal{M}_{reg}}

\newtheorem{theorem}{Theorem}[]

\newtheorem{corollary}[theorem]{Corollary}
\newtheorem{proposition}[theorem]{Proposition}
\newtheorem{remark}[theorem]{Remark}

\theoremstyle{definition}

\newtheorem{definition}{Definition}[section]

\title{Monodromy of the $SL_{2}$ Hitchin fibration}
\author{Laura P. Schaposnik }

\address{Mathematical Institute, 24-29 St. Giles', Oxford, UK  OX1 3LB.}
\email{schaposnik@maths.ox.ac.uk}

\date{\today}

\begin{document}

\markboth{Laura P. Schaposnik}
{Monodromy of the $SL_{2}$ Hitchin fibration}





\maketitle

\begin{abstract}
We calculate the monodromy action on the mod 2 cohomology for $SL(2,\mathbb{C})$ Hitchin systems and give an application of our results in terms of the moduli space of semistable $SL(2,\mathbb{R})$-Higgs bundles. 
\end{abstract}



\section{Introduction}	

Let $\Sigma$ be a Riemann surface, and denote by $K$ its  canonical bundle. An $SL(2,\mathbb{C})$-Higgs bundle as defined by  Hitchin \cite{N1} and  Simpson \cite{simpson} is given by a pair $(E,\Phi)$, where $E$ is a rank 2 holomorphic vector bundle with ${\rm det}(E)=\mathcal{O}$, and the Higgs field $\Phi$ is a   section of ${\rm End}_{0}E\otimes K$, for ${\rm End}_{0}E$   the bundle of traceless endomorphisms. A Higgs bundle is said to be semistable if for any subbundle $F\subset E$ such that $\Phi(F)\subset F\otimes K$ one has ${\rm deg}(F)\leq 0$. When ${\rm deg}(F)< 0$ we say the pair is stable.

Considering  the moduli space $\mathcal{M}$ of $S$-equivalence classes of semistable $SL(2,\mathbb{C})$-Higgs bundles and the map $\Phi \mapsto {\rm det}(\Phi)$, one may define the so-called Hitchin fibration \cite{N1}:
\begin{eqnarray}h:\mathcal{M}&\rightarrow& \mathcal{A}:=H^{0}(\Sigma,K^{2}).\nonumber
\end{eqnarray}
The moduli space $\mathcal{M}$ is homeomorphic to the moduli space of reductive representations of the fundamental group of $\Sigma$ in $SL(2,\mathbb{C})$ via non-abelian Hodge theory \cite{cor},\cite{donald},\cite{N1},\cite{simpson}. 
The  involution on $SL(2,\mathbb{C})$ corresponding to the real form $SL(2,\mathbb{R})$ defines an antiholomorphic involution on the moduli space of representations which, in the Higgs bundle complex structure, is the holomorphic involution $\Theta:(E,\Phi)\mapsto (E,-\Phi)$. In particular, the isomorphism classes of stable Higgs bundles fixed by the involution $\Theta$ correspond to $SL(2,\mathbb{R})$-Higgs bundles  $(E=V\oplus V^{*}, \Phi)$, where $V$ is a line bundle on $\Sigma$, and the Higgs field $\Phi$ is given by
\[\Phi=\left(\begin{array}
              {cc}
0&\beta\\\gamma&0
             \end{array}
\right)\] 
for $\beta:V^{*}\rightarrow V\otimes K$ and $\gamma:V\rightarrow V^{*}\otimes K$. We shall denote by $\mathcal{M}_{SL(2,\mathbb{R})}$ the moduli space of $S$-equivalence classes of semistable $SL(2,R)$-Higgs bundles.

Let $\M$ be the regular fibres of the Hitchin map $h$, and   $\A$ be the regular locus of the base, which is given by quadratic differentials with simple zeros. From \cite {N1} one knows that the smooth fibres are tori of real dimension $6g-6$. There is a section of the fibration fixed by $\Theta$ and this allows us to identify  each fibre with a Prym variety in Section \ref{sec:fib}. The involution $\Theta$ leaves invariant ${\rm det}(\Phi)$ and so defines an involution on each fibre. 

 In Section \ref{sl} we show that the fixed points of $\Theta$ become the elements of order 2 in the abelian variety. 
As a consequence,  the points corresponding to $SL(2,\mathbb{R})$-Higgs bundles give a finite covering of an open set in the base and thus the natural way of understanding the topology of the moduli space of $SL(2,\mathbb{R})$-Higgs bundles is via the monodromy of the covering, the action of the fundamental group as permutations of the fibres. These elements of order 2 can be canonically identified with those flat line bundles over $\Sigma$ which have holonomy in $\mathbb{Z}_{2}$, and which are invariant by the involution $\Theta$. Hence, we can determine the monodromy by looking at the action on mod 2 cohomology.

In Section 4 we recall the construction of the Gauss-Manin connection associated to the Hitchin fibration, and in Section \ref{sec:copeland} we give a short account of the combinatorial study of the monodromy of this connection which appears in \cite{cope1}.                                                                                                                                                                            
 In Sections \ref{sec:group1}-\ref{sec:group3} we study the action on the cohomology of the fibres, and obtain our main result:

\begin{theorem}
 The monodromy group $G\subset GL(6g-6,\mathbb{Z}_{2})$ on the first mod 2 cohomology of the fibres of the Hitchin fibration is given by the group of matrices   of the form
\begin{eqnarray}
\left(\begin{array} {c|c}
I_{2g}&A\\\hline
0&\pi
                \end{array}
\right),
\end{eqnarray}
\noindent  where  \begin{itemize}
                   \item $\pi$ is the quotient action on $\mathbb{Z}_{2}^{4g-5}/(1,\cdots,1)$ induced by the permutation action of the symmetric group $S_{4g-4}$ on $\mathbb{Z}_{2}^{4g-5}$.
\item  $A$  is any $(2g)\times (4g-6)$ matrix with entries in $\mathbb{Z}_{2}$.
                  \end{itemize}

\end{theorem}

  In Section \ref{sec:orbit} we study the orbits of the monodromy group and,  as an application,  we give a new proof of the following result  (\cite{N1},\cite{goldman2}) in Section \ref{sec:component}:

\begin{corollary}  The number of connected components of the moduli space of semistable $SL(2,\mathbb{R})$-Higgs bundles  is $
2\cdot 2^{2g}+2g-3
$. 
\end{corollary}

These connected components are known to be parametrized by the Euler class $k$ of the associated flat $\mathbb{RP}^{1}$ bundle. From our point of view this number $k$ relates to the orbit of a  subset of $1,2,\dots, 4g-4$ with $2g-2k-2$ points under the action of the symmetric group.

\section{The regular fibres of the Hitchin fibration }\label{sec:fib}

\label{sec1}

Consider the Hitchin map $h:\mathcal{M}\rightarrow \mathcal{A}$ given by $(E,\Phi)\mapsto {\rm det}(\Phi)$. From \cite{N1} the map $h$ is proper and surjective, and its regular fibres are abelian varieties. Moreover, for any $a\in \mathcal{A}-\{0\}$ the fibre $\mathcal{M}_{a}$ is connected  \cite{goand}.  An isomorphism class   $(E,\Phi)$ in $\mathcal{M}$ defines  a spectral curve $\rho:S\rightarrow \Sigma$ in the total space of $K$ with equation
\[{\rm det}(\Phi-\eta I)= \eta^{2}+a=0,\]
where $a={\rm det}(\Phi)\in \mathcal{A}$, and $\eta$  the tautological section of $\rho^{*}K$. The curve $S$ is non-singular over $\A$, and the ramification points are given by the intersection of $S$ with the zero section.
Since there is a  natural involution $\tau(\eta)=-\eta$ on $S$,  one can define the Prym variety ${\rm Prym}(S,\Sigma)$ as the set of line bundles $M\in {\rm Jac}(S)$ which satisfy
\begin{eqnarray}\tau^{*} M\cong M^{*}.\label{P1}\end{eqnarray}
Following the ideas of Theorem 8.1 in \cite{N1}, and as seen in Section 2 of \cite{N3}, one has the following description of the regular fibres of the $SL(2,\mathbb{C})$ Hitchin fibration:

\begin{proposition} \cite{N1} The fibre $h^{-1}(a)$ of $\M$ for $a\in  A$  is biholomorphically equivalent to the ${\rm Prym}$ variety of the double covering $S$ of $\Sigma$ determined by the 2-valued differential $\sqrt{a}$.
 \label{fibras}  \end{proposition}

To see this, given a line bundle $L$ on $S$ one may consider the rank two vector bundle $E:=\rho_{*}L$ and construct an associated Higgs bundle as follows.  For an open set $U\subset \Sigma$, multiplication by the tautological section $\eta$ gives a homomorphism
\begin{equation}\eta: H^{0}(\rho^{-1}(U),L)\rightarrow H^{0}(\rho^{-1}(U),L\otimes \rho^{*}K).\label{homo}\end{equation}
By the definition of a direct image sheaf, we have
\[H^{0}(U,E)=H^{0}(\rho^{-1}(U),L).\]
Hence, equation (\ref{homo}) defines the Higgs field
\begin{equation}
\Phi: E\rightarrow  E\otimes K,
\label{homo2}\end{equation}
which is stable since, by construction, it satisfies its characteristic equation $$\eta^{2}+a=0.$$ 

 From \cite{bobi} one has that
\[\Lambda^{2} E= {\rm Nm}(L)\otimes K^{-1}.\]
Recall that any line bundle $M$ on the curve $S$ such that  $\tau^{*}M\cong M^{*}$ satisfies ${\rm Nm}(M)=0$. Hence, for $K^{1/2}$ a choice of square root,   $M=L\otimes \rho^{*} K^{-1/2}$ is in the Prym variety and thus the induced vector bundle $E=\rho_{*}L$ has trivial determinant:
\[\Lambda^{2} E= {\rm Nm}(M\otimes \rho^{*}K^{1/2})\otimes K^{-1}=\mathcal{O}.\]

\label{sec2}

\section{$SL(2,\mathbb{R})$-Higgs bundles} \label{sl}

Recall from \cite{N2} and \cite{simpson88} that if a Higgs bundle $(E,\Phi)$ is stable and  ${\rm deg} ~E = 0$, then there is a unique unitary connection $A$ on $E$, compatible with the holomorphic structure, such that
\begin{eqnarray}
 F_{A}+ [\Phi,\Phi^{*}]=0~\in \Omega^{1,1}(\Sigma, {\rm End}E), \label{2.1}
\end{eqnarray}
 where $F_{A}$ is the curvature of the connection.   The equation (\ref{2.1}) and the holomorphicity condition
$
d_{A}''\Phi=0 \label{hit2}
$
 are known as the \textit{Hitchin equations}, where  $d_{A}''\Phi$ is the anti-holomorphic part of the covariant derivative of $\Phi$. 
Following \cite{N2}, in order to obtain an $SL(2,\mathbb{R})$-Higgs bundle, for $A$ the connection which solves the Hitchin equations (\ref{2.1}), one requires
\begin{eqnarray}
\nabla=\nabla_A+\Phi+ \Phi^{*}
\end{eqnarray}
to have  holonomy in  $SL(2,\mathbb{R})$.  For $\tau$ the anti-linear involution fixing $SL(2,\mathbb{R})$, requiring  $\nabla=\tau(\nabla)$ is equivalent to requiring \begin{eqnarray}
   \nabla_{A}=\tau(\nabla_{A})~{~\rm~and~}~
\Phi+\Phi^{*}=\tau(\Phi+\Phi^{*}).  \label{ecual1}                                                                                                                                                                     
\end{eqnarray}
For $\rho$ the compact form of $SL(2,\mathbb{C})$, one may consider the involution $\rho\tau$. In terms of $\rho\tau$, the equalities (\ref{ecual1}) are given by \begin{eqnarray}
\nabla_{A}=\rho\tau(\nabla_{A}) ~{~\rm~and ~}~
 \Phi-\rho(\Phi)
                =\rho\tau(\rho(\Phi)-\Phi). 
\end{eqnarray}
Hence,  $\nabla$ has holonomy in $SL(2,\mathbb{R})$ if 
$\nabla_{A}$ is invariant under  $\rho\tau$, and $\Phi$ anti-invariant, i.e., if 
the associated $SL(2,\mathbb{C})$-Higgs bundle $(E,\Phi)$ is fixed by the complex linear involution on the Higgs bundle moduli space
\[\Theta:(E,\Phi)\mapsto (E,-\Phi).\] 

Since we consider the moduli space $\mathcal{M}_{SL(2,\mathbb{R})}$ as sitting inside $\mathcal{M}$, two real representations may be conjugate in the space of complex representations but not by real elements.   
Thus, flat $SL(2,\mathbb{R})$ bundles can be equivalent as $SL(2,\mathbb{C})$ bundles but not as $SL(2,\mathbb{R})$ bundles.  The relation between Higgs bundles for real forms and fixed points of involutions has been studied by Garc\'ia-Prada in \cite{R12}.

In order to understand the fixed point set of $\Theta$, note that for a stable  $SL(2,\mathbb{C})$-Higgs bundle $(E,\Phi)$ whose isomorphism class is fixed by $\Theta$, there is an automorphism $\alpha:E\rightarrow E$ whose action by conjugation on $\Phi$ is given by\[\alpha^{-1}\Phi\alpha\mapsto -\Phi.\]
As $\alpha^{2}$ commutes with $\Phi$, by \cite{N1} it acts as $\lambda^{2}$ for some $\lambda\in \mathbb{C}^{*}$. Furthermore, $\alpha$ has constant eigenvalues $\pm \lambda$ and thus $E$ can be decomposed into the corresponding eigenspaces $E=V \oplus V^{*}$. Then, the Higgs field can be expressed as
\begin{equation}\Phi=\left(\begin{array}
              {cc}
a&\beta\\\gamma&-a
             \end{array}
\right)\in H^{0}(\Sigma, {\rm End}_{0}E \otimes K).\label{Higgs}\end{equation}
From this decomposition and the action of $\alpha$, one has that necessarily $\lambda=\pm i$ and $a=0$.
Then, the  involution $\Theta$ acts on $E$ via transformations of the form
\begin{equation}\pm \left(
\begin{array}{cc}
 i&0\\0&-i
\end{array}
\right),\label{sigma}\end{equation}
and, by stability, for $V$ of positive degree one has $\gamma\neq 0$.

\begin{proposition}
 The intersection of the subspace of the Higgs bundle moduli space $\M$ corresponding to $SL(2,\mathbb{R})$-Higgs bundles with the smooth fibres of the Hitchin fibration
\[h:~\M\rightarrow \A,\]
 is given by the elements of order two in those fibres. \label{ppp}
\end{proposition}

\begin{proof} 
Recall from Proposition \ref{fibras} that any $SL(2,\mathbb{C})$-Higgs bundle $(E,\Phi)$ can be obtained by considering the direct image sheaf $\rho_{*}L=E$, for some line bundle $L$ satisfying $M=L\otimes \rho^{*}K^{-1/2}\in {\rm Prym}(S,\Sigma)$.  
From Section \ref{sec1},  the Higgs field $\Phi$ is obtained by multiplication by the tautological section $\eta$. Thus, the fixed points of the involution $\Theta$ correspond to fixed points of $\eta\mapsto -\eta$. Since the involution $\tau:\eta\mapsto -\eta$ acts trivially on the equation of $S$,  the fixed points correspond to the action on the Picard variety and thus $\tau^{*}M \cong M$. Furthermore, from Proposition \ref{fibras} and equation (\ref{P1}) the line bundle also satisfies  $\tau^{*}M \cong M^{*}$. Hence, $M^{2}\cong \mathcal{O}$ and so $M$ is contained in the set of elements of order 2 in the fibre.

Conversely, one can see that an element of order two in the fibres of the $SL(2,\mathbb{C})$ Hitchin fibration induces an $SL(2,\mathbb{R})$-Higgs bundle. 
Indeed, if $\tau^{*}M\cong M$, then $\tau^{*}L\cong L$ for $M=L\otimes \rho^{*}K^{-1/2}$. Thus,   for an invariant open set $\rho^{-1}(U)$ we may decompose the sections of $L$ into
\[H^{0}(\rho^{-1}(U),L)=H^{0}(\rho^{-1}(U),L)^{+}\oplus H^{0}(\rho^{-1}(U),L)^{-},\]   
where the upper indices $+$ and $-$ correspond, respectively, to the invariant and anti-invariant sections. By definition of direct image sheaf, there is a similar decomposition of the sections of $\rho_{*}L$ into
\[H^{0}( U,\rho_{*}L)=H^{0}(U,\rho_{*}L)^{+}\oplus H^{0}(U,\rho_{*}L)^{-}.\]  
Therefore,  we may write $\rho_{*}L=E$ as
\[E=E_{+}\oplus E_{-},\]
where $E_{\pm}$ are line bundles on $\Sigma$. Moreover, since $\Lambda^{2}E\cong \mathcal{O}$ one has $E_{+}\cong E_{-}^{*}$. 

Following Proposition \ref{fibras}, multiplication by the tautological section induces a Higgs field $\Phi: E_{+}\oplus E_{-} \mapsto (E_{+}\oplus E_{-} )\otimes K $. Since $\tau:\eta \mapsto -\eta$, the Higgs field interchanges the invariant and anti-invariant sections and thus maps $E_{+}\mapsto E_{-}\otimes K$ and $E_{-}\mapsto E_{+}\otimes K$. Hence, the induced pair $(E_{+}\oplus E_{-},\Phi)$ is an $SL(2,\mathbb{R})$-Higgs bundle.\end{proof}

Consider now the trivial line bundle  $\mathcal{O}\in {\rm Prym}(S,\Sigma)$. Since its invariant sections are functions on the base, its direct image is given by $\rho_{*}\mathcal{O}=\mathcal{O}\oplus K^{-1}$. Hence,  for $L=\rho^{*}K^{1/2}$ one has 
$$\rho_{*}L=\rho_{*}\rho^{*}K^{1/2}= K^{1/2}\otimes \rho_{*}\mathcal{O}=K^{1/2}\oplus K^{-1/2}.$$
It follows from Section \ref{sec2} that the line bundle $\mathcal{O}\in {\rm Prym}(S,\Sigma)$ has an associated Higgs bundle. 
Note that in this case $\gamma$ maps $K^{1/2}$ to $K^{-1/2}\otimes K=K^{1/2}$, and thus by stability it is nonzero everywhere. By using the $\mathbb{C}^{*}$ action $\gamma$ can be taken to be 1, obtaining a  family of Higgs bundles $(K^{1/2}\oplus K^{-1/2}, \Phi)$, where the Higgs  field $\Phi$ is constructed via Proposition \ref{fibras}: 
\[\Phi=\left(\begin{array}{cc}
              0&a\\1&0
             \end{array}
\right),~{~\rm~for~}~ a\in H^{2}(\Sigma,K^{2}).\]
This family of $SL(2,\mathbb{R})$-Higgs bundles, which defines an origin in the fibres of $\M$, is the Teichm\"uller space  (see \cite{N1} for details).

\begin{remark}
 Proposition \ref{ppp} can be shown to be true for the adjoint group of any split real form \cite{laura}.
\end{remark}

\begin{definition}\label{def:p}
 We denote by $P[2]$ the elements of order 2 in ${\rm Prym }(S,\Sigma)$,  for a fixed curve $S$.
\end{definition}

\begin{proposition} \label{prop:mono2}
The space $P[2]$ is equivalent to the mod 2 cohomology $H^{1}({\rm Prym}(S,\Sigma),\mathbb{Z}_{2})$.
\end{proposition}

\begin{proof}
Recall that the Prym variety  is given by $$H_{1}({\rm Prym}(S,\Sigma), \mathbb{R})/H_{1}({\rm Prym}(S,\Sigma), \mathbb{Z}),$$ and thus for some lattice $\wedge$ we can write ${\rm Prym}(S,\Sigma)\cong \mathbb{R}^{6g-6}/\wedge$. In particular, one has $\wedge \cong H_{1}({\rm Prym}(S,\Sigma), \mathbb{Z})$. Hence, $P[2]$ is given by equivalence classes in $\mathbb{R}^{6g-6}$ of points $x$ such that $2x \in \wedge$, this is,  by $\frac{1}{2}\wedge$ modulo $\wedge$. Thus, since $\wedge$ is torsion free, \[P[2]\cong \wedge/2\wedge\cong H_{1}({\rm Prym}(S,\Sigma),\mathbb{Z}_{2}).\]
Moreover, $H^{1}({\rm Prym}(S,\Sigma),\mathbb{Z}_{2})\cong {\rm Hom}(H_{1}({\rm Prym}(S,\Sigma),\mathbb{Z}),\mathbb{Z}_{2})$ and so
\[H^{1}({\rm Prym}(S,\Sigma),\mathbb{Z}_{2})\cong {\rm Hom}(\wedge,\mathbb{Z}_{2})\cong \wedge/2\wedge\cong P[2],\]
whence proving the proposition. \end{proof}

\section{The monodromy action}\label{sec:mono}

The homologies of ${\rm Prym}(S,\Sigma)$ of the Hitchin fibration $\M\rightarrow \A$ carry a canonical connection, the \textit{Gauss-Manin connection}, and its holonomy is the \textit{monodromy action}. We shall recall here the construction of this connection and its holonomy.  

 Consider a fibration $p:Y \rightarrow B$ which is locally trivial, i.e. for any point $b\in B$ there is an open neighbourhood $U_{b}\in B$ such that $p^{-1}(U_{b})\cong U_{b}\times Y_{b}$, where $Y_{b}$ denotes the fibre at $b$.   The $n$th homologies of the fibres $Y_{b}$ form a locally trivial vector bundle over $B$, which we denote $\mathcal{H}_{n}(B)$. This bundle carries a canonical flat connection, the \textit{Gauss-Manin connection}.

To define this connection we  identify the fibres of $\mathcal{H}_{n}(B)$ at nearby points $b_{1},b_{2}\in B$, i.e., $H_{n}(Y_{b_{1}})$ and $H_{n}(Y_{b_{2}})$. Consider $N\subset B$ a contractible open set which includes $b_{1}$ and $b_{2}$.  The inclusion of the fibres $Y_{b_{1}}\hookrightarrow p^{-1}(N)$ and $Y_{b_{2}}\hookrightarrow p^{-1}(N)$ are homotopy equivalences,  and hence we obtain an isomorphism between the homology of a fibre over a point in a contractible open set  $N$ and $H_n(p^{-1}(N))$: \[H_{n}(Y_{b_{1}})\cong H_{n}(p^{-1}(N))\cong H_{n}(Y_{b_{2}}).\]
This means that the vector bundle $\mathcal{H}_n(B)$ over $B$  has a flat connection, the Gauss-Manin connection.   The \textit{monodromy} of $\mathcal{H}_n(B)$   is the holonomy of this connection, i.e.,   for a choice of base point $b_{0}$, a homomorphism $\pi_{1}(B)\rightarrow {\rm Aut}(H_{n}(Y_{b_{0}}))$.

As each regular fibre of $\M\rightarrow \A$ is a torus,  the monodromy action preserves the integral lattice $\wedge$ introduced in Proposition \ref{prop:mono2}, and for the base point $b_{0}$ defining the fixed curve $S$ in Definition \ref{def:p}, it is generated by the action of $\pi_{1}(\A)$ on $H^{1}( {\rm Prym}(S,\Sigma),\mathbb{Z})$. Hence,  there is an induced action on the quotient
\begin{eqnarray}
 P[2]=H_{1}({\rm Prym}(S,\Sigma), \mathbb{Z})/2\wedge.
\end{eqnarray}
The monodromy of the Hitchin fibration induces an action on the elements of order 2 in the regular fibres, which are equivalence classes of line bundles of order 2 invariant by the involution. This is precisely $P[2]$ and so we can calculate the monodromy by cohomological means.

\section{A combinatorial approach to monodromy for $SL(2,\mathbb{C})$}\label{sec:copeland}

 The generators and relations of the monodromy action of the  $SL(2,\mathbb{C})$ Hitchin fibration for hyperelliptic surfaces were studied from a combinatorial point of view by Copeland in \cite{cope1}. We shall dedicate this section to give a short account of \cite{cope1}. His main result for hyperelliptic curves, which may be extended by \cite{walker}  to any compact Riemann surface, is the following:

\begin{theorem}[\cite{cope1,walker}]
 To each compact Riemann surface $\Sigma$ of genus greater than 2, one may associate a graph $\check{\Gamma}$ with edge set $E$ and a skew bilinear pairing  $<e , e'>$ on edges  $e,e' \in \mathbb{Z}[E]$ such that
\begin{enumerate}
 \item[(i)] the monodromy representation of $\pi_{1}(\A)$ acting on $H_{1}({\rm Prym}(S,\Sigma),\mathbb{Z})$ is generated by elements $\sigma_{e}$ labelled by the edges $e\in E$,
\item[(ii)] one can define an action of $\pi_{1}(\A)$ on $e'\in \mathbb{Z}[E]$ given by
\[\sigma_{e}(e')=e'-<e',e>e,\]
\item[(iii)] the monodromy representation of the action of $\pi_{1}(\A)$ on $H_{1}({\rm Prym}(S,\Sigma),\mathbb{Z})$ is a quotient of this module $\mathbb{Z}[E]$.
\end{enumerate}\label{copeland}
\end{theorem}

In order to construct the graph $\check{\Gamma}$ Copeland looks at the particular case of $\Sigma$ given by the non-singular compactification of the zero set of $y^{2}=f(x)=x^{2g+2}-1$. Firstly, by considering $\omega \in \mathcal{A}$ given by
\[\omega=(x-2\zeta^{2})(x-2\zeta^{4})(x-2\zeta^{6})(x-2\zeta^{8})\prod_{9\leq j\leq 2g+2}(x-2\zeta^{j})\left(\frac{dx}{y}\right)^{2},\]
for $\zeta=e^{2\pi i /2g+2}$, it is shown in \cite{cope1} how interchanging two zeros of the differential provides information about the generators of the monodromy. Then, by means of the ramification points of the surface,  a dual graph to $\check{\Gamma}$ for which each zero of $\omega$ is in a face could be constructed.  Copeland's analysis extends  to any element in $\A$.

 Following  \cite{cope1}, we consider the graph $\check{\Gamma}$ whose $4g-4$ vertices are given by the ramification divisor of $\rho:S\rightarrow \Sigma$, i.e., the zeros of $a={\rm det}(\Phi)$. For genus $g=3,5,$ and $10$, the graph $\check{\Gamma}$ constructed by Copeland in \cite{cope1} is given by:

\begin{figure}[htbp]
\centering \epsfig{file=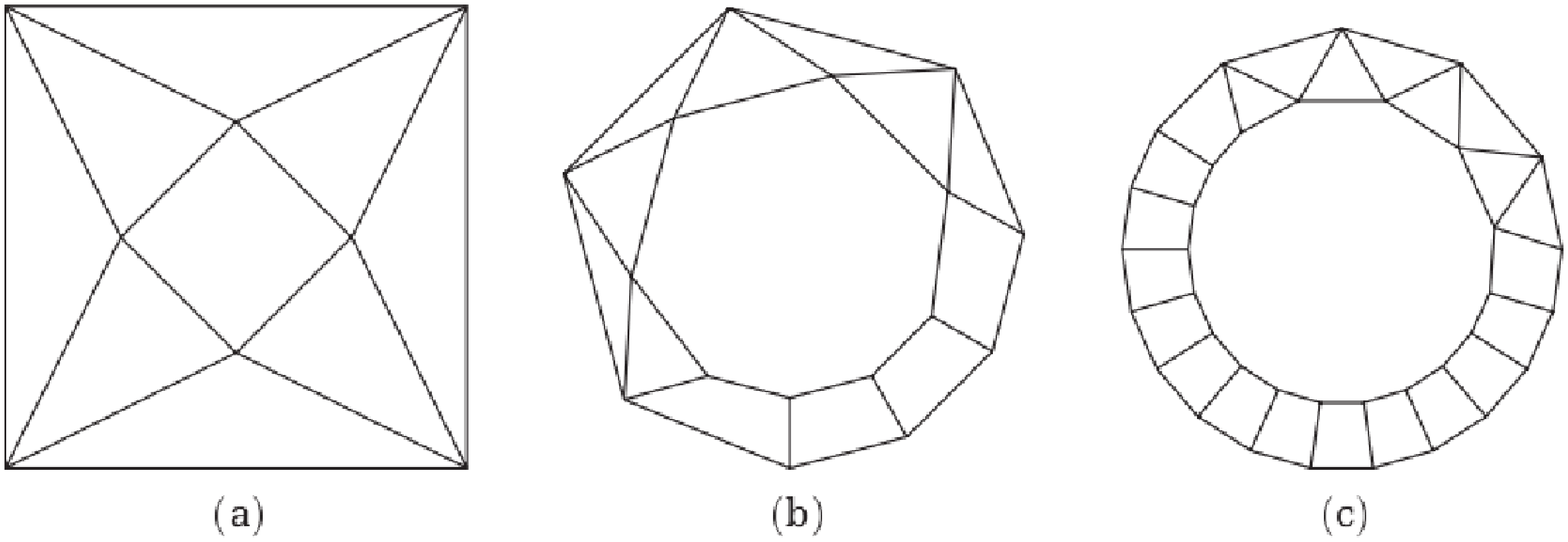, width=9 cm}
\caption{}
\label{grafo4}
\end{figure}
For  $g> 3$ the graph $\check{\Gamma}$ is given by a ring with 8 triangles next to each other, $2g-6$ quadrilaterals and $4g-4$ vertices, and  we shall label its edges as in \cite{cope1}:
\begin{figure}[htbp]
\centering \epsfig{file=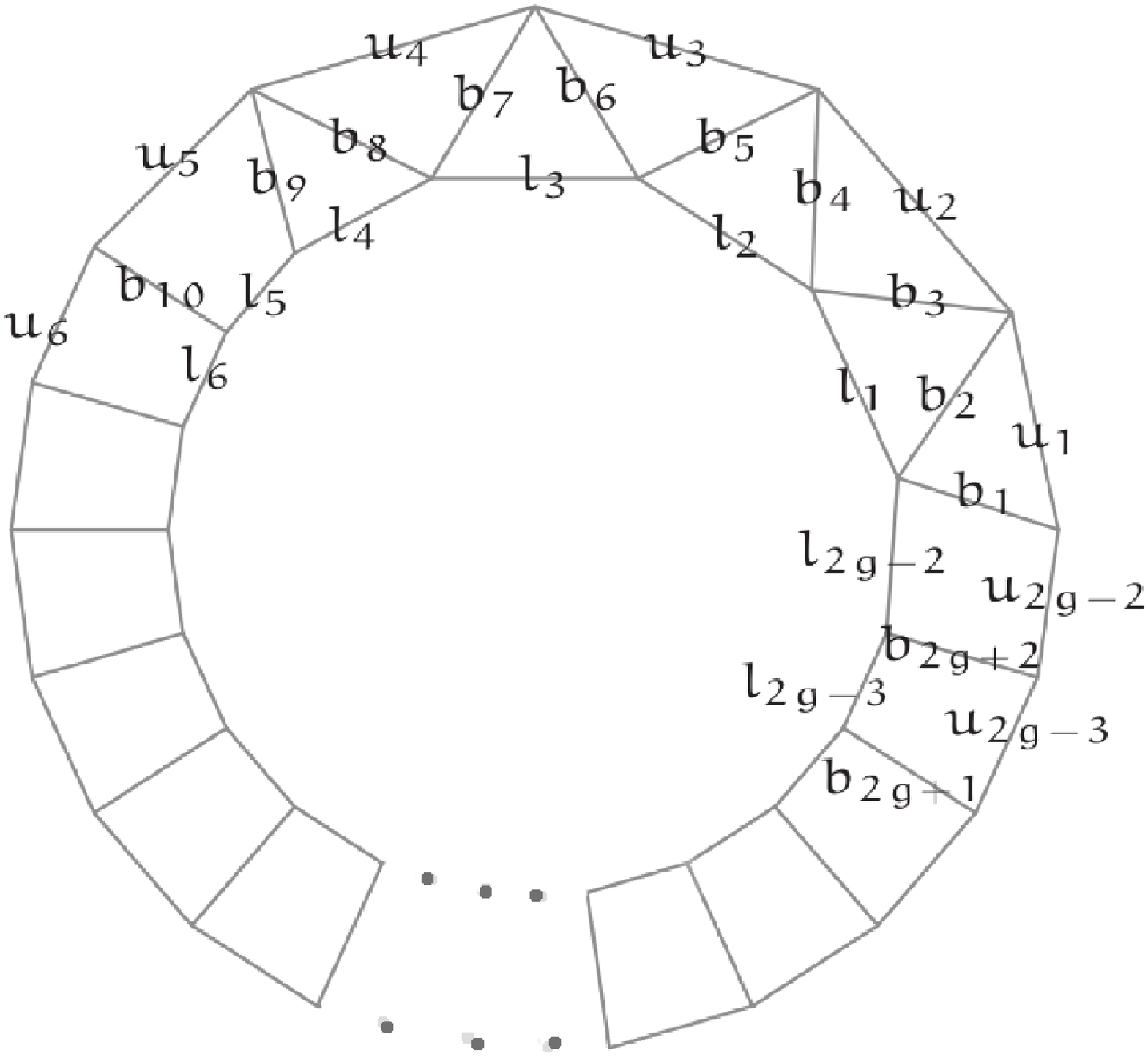, width=7 cm}
\caption{} 
\label{annulus}
\end{figure}

Considering the lifted graph of $\check{\Gamma}$ in the curve $S$ over $\Sigma$, Copeland could show the following in Theorem 16  and Proposition 21 \cite{cope1}:
\begin{proposition}\cite{cope1}\label{pro3}
If $E$ and $F$ are respectively the edge  and face sets of $\check{\Gamma}$, then  there is an induced homeomorphism
\begin{eqnarray}
 {\rm Prym}(S,\Sigma)\cong \frac{\mathbb{R}[E]}{\left(\mathbb{R}[F]\cap\frac{1}{2}\mathbb{Z}[E]\right)},\label{prince}
\end{eqnarray}
where the inclusion $\mathbb{R}[F] \subset \mathbb{R}[E]$ is defined by the following relations between the edges $\check{\Gamma}$ and the boundaries $\tilde{x}_{1},\ldots,\tilde{x}_{4}$ of the faces:
\begin{eqnarray}
 ~\tilde{x}_{1}&=&\sum_{i=1}^{2g-2}l_{i} ~;\nonumber\\ ~\tilde{x}_{2}&=&\sum_{i=1}^{2g-2}u_{i}~;\nonumber\\
~ \tilde{x}_{3}&=&~\sum_{even \geq  6}u_{i}-\sum_{odd \geq 5}l_{i}+\sum_{i=1}^{2g+2}b_{i}~; \nonumber\\
~\tilde{x}_{4}&=&~l_{1}+l_{3}-u_{2}-u_{4}+\sum_{odd}u_{i}-\sum_{even}l_{i}+\sum_{i=1}^{2g+2}b_{i}.\nonumber
\end{eqnarray}
  \label{principal}\label{prim}
\end{proposition} 
Note that $\mathbb{R}[F] \cap \frac{1}{2}\mathbb{Z}[E]$ can be understood by considering the following sum:
\[\tilde{x}_{3}+\tilde{x}_{4}+\tilde{x}_{1}-\tilde{x}_{2}= 2\left(l_{1}+l_{3}-u_{2}-u_{4}+\sum_{i=1}^{2g+2}b_{i}\right)=2 \tilde{x}_{5}.\]
Although the summands above are not individually in $\mathbb{R}[F] \cap \frac{1}{2}\mathbb{Z}[E]$, when summed they satisfy
\[\tilde{x}_{5}\in \mathbb{R}[F] \cap \frac{1}{2}\mathbb{Z}[E].\]

\begin{remark}\label{remar}
  From Proposition \ref{fibras} and Proposition \ref{pro3} one has that the fibres of $\M$ are isomorphic to the quotient of $\mathbb{R}[E]$ by the space spanned by $\{\tilde{x}_{1},\ldots, \tilde{x}_{5}\}.$
\end{remark}

\begin{remark} 
For $g=2$ it is known that  $\pi_{1}(\A)\cong \mathbb{Z}\times \pi_{1}(S_{6}^{2})$,  where $S_{6}^{2}$ is the sphere $S^{2}$ with $6$ holes (e.g. \cite{cope1}).
\end{remark}

\section{The action on $\mathbb{Z}_{2}[E]$}\label{sec:group1}\label{sec:mono2}

From Proposition \ref{prop:mono2},   the monodromy action for $SL(2,\mathbb{R})$-Higgs bundles can be understood through the monodromy of the Gauss-Manin connection on the mod 2 cohomology of the fibres of the Hitchin fibration. This is the action on $P[2]$, and thus it is convenient to first give a combinatorial description of this space.

\begin{proposition}\label{Pmon}
 The space $P[2]$ is   the quotient of $\mathbb{Z}_{2}[E]$ by the subspace generated by $x_{1}, x_{2}, x_{4}$ and $x_{5}$, where
 \begin{eqnarray}
~x_{1}&:=&\sum_{i=1}^{2g-2}l_{i};~\nonumber\\ 
~x_{2}&:=&\sum_{i=1}^{2g-2}u_{i};~ \nonumber\\
~x_{4}&:=&l_{1}+l_{3}+u_{2}+u_{4}+\sum_{odd}u_{i}+\sum_{even}l_{i}+\sum_{i=1}^{2g+2}b_{i}~;~{~~}~\nonumber\\
x_{5}&:=& l_{1}+l_{3}+u_{2}+u_{4}+\sum_{i=1}^{2g+2}b_{i}.\label{cuarta}\nonumber
\end{eqnarray}

\end{proposition}

\begin{proof}
Following the notation of Proposition \ref{prop:mono2}, from Proposition \ref{prim} one may write ${\rm Prym}(S,\Sigma)\cong \mathbb{R}^{6g-6}/\wedge$, where $$\wedge:=\frac{\frac{1}{2}\mathbb{Z}[E]}{\mathbb{R}[F]\cap\frac{1}{2}\mathbb{Z}[E]}.$$ Then, the result follows from Remark \ref{remar}, since over $\mathbb{Z}_{2}$, the equations for $\tilde{x}_{1},\tilde{x}_{2}, \tilde{x}_{3}, \tilde{x}_{4}$ and  $\tilde{x}_{5}$  are equivalent to the relations $x_{1},x_{2},X_{4}$ and $x_{5}$.
\end{proof}

Since the space $P[2]$ is given by the quotient of $\mathbb{Z}_{2}[E]$ by the subspace generated by $x_{1}, x_{2}, x_{4}$ and $x_{5}$, we shall dedicate this section to understand the action of the generators from Theorem \ref{copeland} on $\mathbb{Z}_{2}[E]$.

\begin{definition}
 Let $C_{1}$ be the the space  of 1-chains for a subdivision of the annulus in Figure \ref{annulus}, i.e., the space $\mathbb{Z}_{2}[E]$ spanned by the edge set $E$ over $\mathbb{Z}_{2}$. The boundary  map $\partial$ to the space $C_{0}$ of 0-chains (spanned by the vertices of $\check{\Gamma}$ over $\mathbb{Z}_{2}$) is defined   on an edge $e\in C_{1}$ with vertices $v_{1}, ~v_{2}$ as $\partial e= v_{1}+v_{2}$. 
\end{definition}

For $\Sigma^{[4g-4]}$ be the configuration space of $4g-4$ points in $\Sigma$, there is a natural map  $p:\A\rightarrow \Sigma^{[4g-4]}$ which takes a quadratic differential to its zero set. Furthermore, $p$ induces the following maps
\[\pi_{1}(\A) \rightarrow \pi_{1}(\Sigma^{[4g-4]}) \rightarrow S_{4g-4}~,\]
where $S_{4g-4}$ is the symmetric group of $4g-4$ elements. Thus there is  a natural  permutation action on  $C_{0}$,  and Copeland's generators in $\pi_{1}(\A)$ map to transpositions in $S_{4g-4}$. Concretely, a generator   $\sigma_{e}$ labelled by the edge $e$ as in Theorem \ref{copeland} acts on another edge $x$ by
\begin{equation}\sigma_{e}(x)= x ~+ <x,e>e,\label{generator}\end{equation}
where $<\cdot,\cdot>$ is the intersection pairing. As this pairing is skew over $\mathbb{Z}$, for any edge $e$ one has $<e,e>=0$.

\begin{definition}
We denote by $G_{1}$ the group of transformations of $C_{1}$ generated by $\sigma_{e}$, for $e\in E$. 
\end{definition}
 
\begin{proposition} The group $G_{1}$ acts trivially on $Z_{1}= \ker (\partial:C_{1}\rightarrow C_{0})$.   \label{trivial}
\end{proposition}  

\begin{proof} 
Consider  $a\in C_{1}$ such that $\partial a =0$, i.e., the edges of $a$ have vertices which occur an even number of times. By definition,  $\sigma_{e}\in G_{1}$ acts trivially on $a$ for any edge $e\in E$  non adjacent to $a$.  Furthermore, if $e\in E$ is adjacent to $a$, then $\partial a=0$ implies that an even number of edges in $a$ is adjacent to $e$, and thus the action $\sigma_{e}$ is also trivial on $a$.
\end{proof}

We shall give an ordering to the vertices in $\check{\Gamma}$ as in the figure below, and denote by  $E'\subset E$ the set of dark edges:
\begin{figure}[htbp]
\centering \epsfig{file=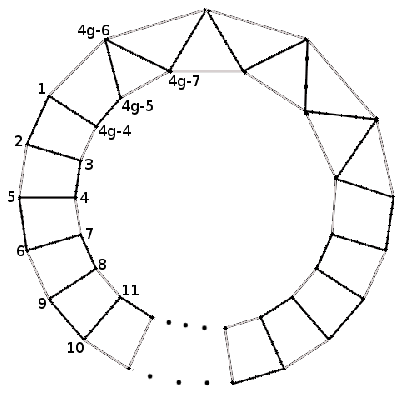, width=6 cm}
\caption{}
\end{figure}

\begin{definition}
 For $(i,j)$ the edge between the vertices $i$ and $j$, the set  $E'$ is given by the edge $e_{4g-4}:=(4g-4,1)$ together with the natural succession of edges $e_{i}:=(i, i+1)$ for $i=1,\ldots,4g-5$.
\end{definition}

\begin{proposition}
 The reflections labelled by the edges in $E' \subset E$ generate a subgroup $S'_{4g-4}$ of $G_{1}$ isomorphic to the symmetric group $S_{4g-4}$. \label{propi1}
\end{proposition}

\begin{proof}
One needs to check that the following properties characterizing generators of the symmetric group apply to the reflections labelled by $E'$: 
\begin{enumerate}
 \item[(i)] $\sigma_{e_{i}}^{2}=1$ for all $i$,
 \item[(ii)] $\sigma_{e_{i}}\sigma_{e_{j}}=\sigma_{e_{j}}\sigma_{e_{i}}$ if $j\neq i\pm 1 $,
 \item[(iii)] $(\sigma_{e_{i}}\sigma_{e_{i+1}})^{3}=1.$
\end{enumerate}

By equation (\ref{generator}), it is straightforward to see that the properties (i) and (ii) are satisfied by $\sigma_{e_{i}}$ for all $e_{i}\in E$.
In order to check (iii) we shall consider different options for edges adjacent  to $e_{i}$ and $e_{i+1}$ when  $e_{i},e_{i+1}\in E'$.
Let $c_{1},c_{2},\cdots, c_{n}\in E$ be the  edges adjacent to $e_{i}$ and $e_{i+1}$, where $n$ may be $5,6$ or $7$. Taking the basis $\{c_{1}, \cdots, c_{n},e_{i},e_{i+1}\}$, the action  $\sigma_{e_{i}}\sigma_{e_{i+1}}$ is given by  matrices $B$ divided into blocks in the following manner:
\begin{equation}
B:=\left(
\begin{array}{ccc|cc}
   &  &  &   0 & 0 \\
   & I_{n} &   & \vdots & \vdots \\
   &  &  &   0 & 0 \\\hline
 a_{1} &   \cdots   & a_{n} & 0 & 1 \\
 b_{1} &  \cdots     &b_{n}  & 1 & 1
\end{array}
\right),\nonumber
\end{equation}
 where the entries $a_{i}$ and $b_{j}$ are $0$ or $1$, depending on the number of common vertices with the edges and their locations. Over $\mathbb{Z}_{2}$ one has that $B^{3}$ is the identity matrix
and so  property  (iii) is satisfied for all edges $e_{i}\in E'$. \end{proof}

The subgroup $S'_{4g-4}$ preserves $\mathbb{Z}_{2}[E']$, and the boundary $\partial: C_{1}\rightarrow C_{0}$ is compatible with the action of $S_{4g-4}'$ on $C_{1}$ and $S_{4g-4}$ on $C_{0}$.
Thus, from  Proposition \ref{propi1}, there is a natural homomorphism  $$\alpha:G_{1}\longrightarrow S_{4g-4}~,$$ which is an isomorphism when restricted to $S'_{4g-4}$. 

\begin{definition}
We shall denote by  $N$ the kernel  of $\alpha: G_{1}\longrightarrow S_{4g-4}$.
\end{definition}

From the above analysis, in terms of $N$ one has that
\begin{eqnarray}G_{1}=N \ltimes S'_{4g-4}~.\label{semid}\end{eqnarray}
 Any element $g\in G_{1}$ can  be expressed uniquely as $g=s \cdot h$, for $h\in N \subset G_{1}$ and $s \in S'_{4g-4}\subset G_{1}$. The  group action on $C_{1}$ may then be expressed as
\begin{equation}
 (h_{1}s_{1})(h_{2}s_{2})=h_{1}s_{1}h_{2}s_{1}^{-1}s_{1}s_{2}~,\label{conj}
\end{equation}
for $s_{1},s_{2}\in S_{4g-4}'$ and $h_{1},h_{2}\in N$.  In order to understand the subgroup $N$ we make the following definition.
\begin{definition}
 Let $E_{0}:=E-E'$,  and let $\Delta_{e}\in C_{1}$   denote the boundary of the square or triangle adjacent to the edge $e\in E_{0}$  in $\check{\Gamma}$.
\end{definition}

 Note that each edge $e\in E_{0}$ is contained in only one of such boundaries. 

\begin{definition}
 We define  $\tilde{E}_{0}:=\{\Delta_{e}~{\rm for}~e\in E_{0}\}$. 
\end{definition}
From Proposition \ref{trivial} the boundaries  $\Delta_{e}$ are acted on trivially by $G_{1}$. By comparing the action labelled by the  edges in $E_{0}$ with the action of the corresponding transposition in $S'_{4g-4}$ one can find out which elements are in $N$.

\begin{definition}
  We shall denote by $\sigma_{(i,j)}$ the action on $C_{1}$ labeled by the edge $(i,j)$, and let $s_{(i,j)}$ be the element in $S'_{4g-4}$ such that $\alpha(\sigma_{(i,j)})=\alpha(s_{(i,j)})\in S_{4g-4}$, where $\alpha(s_{(i,j)})$ interchanges the vertices of $(i,j)$.
\end{definition}

 Then, we have the following possible relations between the action of elements  $\sigma_{(i,j)}$ and $s_{(i,j)}$ :

\begin{itemize}
\item For a triangle with vertices  $i,i+1,i+2$, the generators of $S'_{3}\subset S'_{4g-4}$ are  given  by $\sigma_{(i,i+1)}$ and $\sigma_{(i+1,i+2)}$. Then,   $(i,i+1),(i+1,i+2) \in E'$ and the action labelled by  the edge $(i,i+2)\in E_{0}$ can be written as:
\[s_{(i,i+2)}=\sigma_{(i+1,i+2)}\sigma_{(i,i+1)}\sigma_{(i+1,i+2)}~.\]

\item For a square with vertices $i,i+1,i+2, i+3$, the generators of $S'_{4} \subset S'_{4g-4}$ are given by  $\sigma_{(i,i+1)}, \sigma_{(i+1,i+2)}$ and $\sigma_{(i+2,i+3)}$. In this case one has that the edges $(i,i+1),(i+1,i+2),(i+2,i+3)\in E'$ and the action labelled by the edge $(i,i+3)\in E_{0}$ can be written as:
\[s_{(i,i+3)}=\sigma_{(i+2,i+3)}\sigma_{(i+1,i+2)}\sigma_{(i,i+1)}\sigma_{(i+1,i+2)}\sigma_{(i+2,i+3)}~.\]
\end{itemize}

\begin{theorem} The action of  $\sigma_{e}$ labelled by $e\in E$ on $x\in C_{1}$ is given by
\[\sigma_{e} (x) =  h_{e}\cdot s_{e} ~(x)~, \]
where $s_{e}\in S'_{4g-4}$ is the element which maps under $\alpha$ to the transposition of the two vertices of $e$, and  the action of $h_{e} \in N$ is given by
\begin{equation}
  h_{e}(x)=\left\{ \begin{array}{ccc}
x&{\rm if}&e \in E'~,\\
x+<e,x>\Delta_{e}&{\rm if}& e\in E_{0}~.
                \end{array}\right. \label{funh}
\end{equation}\label{teomuy}
\end{theorem}

\begin{proof}For $e\in E'$, one has $\sigma_{e}\in S'_{4g-4}$. Furthermore, one can see that the action of $\sigma_{e}$ labelled by $e\in E_{0}$ on an adjacent edge $x$ is given by 
\[\sigma_{e}(x)=x+<x,e>e_{i}=s_{e}(x)+\Delta_{e}~.\]
From Proposition \ref{trivial}, the boundary $\Delta_{e}$ is acted on trivially by $G_{1}$ and thus the above action is given by
\[ h_{e}\cdot  s_{e} (x)
=h_{e}(\sigma_{e}(x)+\Delta_{e})
=h_{e}(x)+<x,e_{i}>e_{i}+\Delta_{e}
=\sigma_{e}(x).
\]
\end{proof}

\begin{remark}\label{remi2} Note that for $e,e'\in E_{0}$, the maps $h_{e}$ and $h_{e'}$ satisfy
\begin{eqnarray}
 h_{e}h_{e'}(x)
&=& x+<x,e'>\Delta_{e'}+<e,x>\Delta_{e}~. \label{note}
\end{eqnarray}
\end{remark}

\section{A representation of the action on $\mathbb{Z}_{2}[E]$} \label{sec:group2}

In order to construct a representation for the action of $\pi_{1}(\A)$ on $C_{1}$, we shall begin by studying the image $B_{0}$ and the kernel $Z_{1}$ of $\partial: C_{1} \rightarrow C_{0}$.

\begin{definition}
 For $y=(y_{1},\dots, y_{4g-4})\in C_{0}$,  we define the linear map $f:C_{0}\rightarrow \mathbb{Z}_{2}$   by
\begin{equation}
 f(y)=\sum_{i=1}^{4g-4}y_{i} ~.\label{effe}
\end{equation}
\end{definition}

\begin{proposition}
The image  $B_{0}$ of $\partial: C_{1} \rightarrow C_{0}$ is formed by elements with an even number of $1$'s, i.e.,  $B_{0}=\ker f.$\label{propi2}
\end{proposition}

\begin{proof}It is clear that $B_{0}\subset \ker f$. In order to check surjectivity we consider the edges  $e_{i}=(i,i+1)\in E'$, for $i=1, \dots,4g-5$. 
Given the elements $R^{k}\in B_{0}$ for $k=2,\cdots,4g-4$ defined as
 \begin{eqnarray}
R^{2}&:=&\partial e_{1}  =(1,1,0,\dots,0)~,\nonumber\\
&\vdots&\nonumber\\
R^{k}&:=& R^{k-1}+\partial e_{k-1}=(1,0,\dots,0,1,0,\dots,0),\nonumber
 \end{eqnarray}
one may generate any distribution of an even number of $1$'s. Hence,  $B_{0}={\rm span}\{R^{k}\}$ which is the kernel of $f$. \end{proof}

\begin{proposition} The dimensions of the image $B_{0}$ and the kernel $Z_{1}$ of the derivative $~\partial: C_{1} \rightarrow C_{0}~$ are, respectively,  $5g-5$ and $2g+3$.
\end{proposition}

\begin{proof} From Prop. \ref{propi2} one has that $\dim(B_{0})=\dim(\ker f)=4g-5$. Furthermore, as $\dim C_{1}=\dim Z_{1} + \dim B_{0},$ the kernel $Z_{1}$ of $\partial$ has dimension $2g+3$. \end{proof}

Note that  $x_{1},x_{2},x_{4}\in Z_{1}$ and $x_{5}\notin Z_{1}$. From a homological viewpoint one can see that $x_{4}$ and  $\Delta_{e}~{\rm for}~e\in E_{0}$ form a basis for the kernel $Z_{1}$. We can extend this to a basis of $C_{1}$ by taking the edges $$\beta':=   \{e_{i}=(i,i+1) ~{\rm for}~ 1\leq i\leq 4g-5\}\subset E'~,$$ whose images under $\partial$ form a basis for $B_{0}$, and hence a basis for a complementary subspace $V$ of $Z_{1}$.

\begin{definition}
  We   denote by $\beta:=\{\beta_{0},\beta'\}$ the basis of $C_{1}$, for $\beta_{0}:=\{\tilde{E}_{0},x_{4}\}$.
\end{definition}

In order to generate the whole group $G_{1}$,  we shall study the action of $S'_{4g-4}$  by conjugation on $N$. Considering the basis $\beta$  one may construct a matrix representation of the maps $h_{e_{i}}$ for $e_{i} \in  E=\{E_{0},E'\}$. 

\begin{proposition} For $e\in E$, the matrix $[h_{e}]$ associated to $h_{e}$ in the basis $\beta$  is given by
\begin{equation}
 [h_{e}]=\left(\begin{array} {c|c}
  I_{2g+3}&A_{e}\\ \hline
0&I_{4g-5}\\
 \end{array}
\right)~,\nonumber
\end{equation}
\noindent  where the $(2g+3) \times (4g-5)$ matrix $A_{e}$ satisfies one of the following:
\begin{itemize}
 \item it is the zero matrix for $e\in E'$,
 \item it has only four non-zero entries  in the intersection of the row corresponding to $\Delta_{e}$ and the columns corresponding to an adjacent edge of $e$, for $e\in E_{0}-\{u_{5},l_{6}\}$,
 \item it has three non-zero entries  in the intersection of the row corresponding to $\Delta_{e}$ and the columns corresponding to an adjacent edge of $e$, for $e=u_{5},l_{6}$.
\end{itemize}
\end{proposition}

\begin{proof}As we have seen before, for $e\in E_{0}$ the map $h_{e}$ acts as the identity on the elements of $\beta_{0}$. Furthermore, any edge $e\in E_{0}-\{u_{5},l_{6}\}$ is adjacent to exactly four edges in $\beta'$. In this case $h_{e}$ has exactly four non-zero elements in the intersection of the row corresponding to $\Delta_{e}$ and the columns corresponding to edges in $\beta'$ adjacent to $e$. In the case of $e=u_{5},l_{6}$, the edge $e$ is adjacent to exactly 3 edges in $\beta'$ and thus $h_{e}$ has only 3 non-zero entries. \end{proof}

Recall that $S'_{4g-4}$ preserves the space spanned by $E'$, and hence also the subspace $V$ spanned by $\beta'$, and acts trivially on $Z_{1}$. In the basis $\beta$ the action of an element $s\in S'_{4g-4}$ has a matrix representation given by
\begin{equation}
 [s]=\left(\begin{array} {c|c}
  I_{2g+3}&0\\ \hline
0&\pi_{s}\\
 \end{array}
\right)~,\nonumber
\end{equation}
where $\pi$ is the permutation action corresponding to $s$. Hence,  for $f\in E_{0}$ we may construct the matrix for a conjugate of $h_{f}$ as follows:
\begin{eqnarray}
[s] [h_{f}] [s]^{-1}&=&
 \left(\begin{array} {c|c}
  I_{2g+3}&0\\ \hline
0&\pi_{s}\\
 \end{array}
\right) 
\left(\begin{array} {c|c}
  I_{2g+3}&A_{f}\\ \hline
0&I_{4g-5}\\
 \end{array}
\right) 
\left(\begin{array} {c|c}
  I_{2g+3}&0\\ \hline
0&\pi_{s}^{-1}\\
 \end{array}
\right)
=\left(\begin{array} {c|c}
  I_{2g+3}&A_{f}\pi_{s}^{-1}\\ \hline
0&I_{4g-5}\\
 \end{array}
\right).\nonumber
\end{eqnarray}

\begin{proposition} The normal subgroup $N \subset G_{1}$  consists of all matrices of the form
\begin{equation}
H=\left(\begin{array} {c|c}
  I_{2g+3}&A\\ \hline
0&I_{4g-5}\\
 \end{array}
\right)\nonumber
\end{equation}  

\noindent where $A$ is any matrix whose rows corresponding to $\Delta_{e}$ for $e\in E_{0}-\{u_{5},l_{6}\}$ have an even number of $1$'s, the row corresponding to $x_{4}$ is zero and the rows corresponding to $\Delta_{u_{5}},\Delta_{l_{6}}$ have any distribution of $1$'s.   
\label{muy}
\end{proposition}

\begin{proof} Given $e_{i} \in E_{0}-\{u_{5},l_{6}\}$,  the matrix $A_{e_{i}}$ has only four non-zero entries in the row corresponding to $\Delta_{e_{i}}$. Thus, for $g>2$ , there exist  elements $s_{1},s_{2}\in S'_{4g-4}$ with associated permutations $\pi_{1}$ and $\pi_{2}$ such that the matrix
$\tilde{A}_{e_{i}}:=A_{e_{i}}\pi_{1}^{-1}+A_{e_{i}}\pi_{2}^{-1}$ has only two non-zero entries in the row corresponding to $\Delta_{e_{i}}$, given by the vector $R^{5}$  defined in the proof of Proposition \ref{propi2}.
Furthermore, by Remark \ref{remi2}, we have
\begin{equation}[s_{1}h_{e_{i}}s^{-1}_{1}s_{2}h_{e_{i}}s^{-1}_{2}]
=
\left(\begin{array} {c|c}
  I_{2g+3}&A_{e_{i}}\pi_{1}^{-1}\\ \hline
0&I_{4g-5}\\
 \end{array}
\right)\left(\begin{array} {c|c}
  I_{2g+3}&A_{e_{i}}\pi_{2}^{-1}\\ \hline
0&I_{4g-5}\\
 \end{array}
\right)
=\left(\begin{array} {c|c}
  I_{2g+3}&\tilde{A}_{e_{i}}\\ \hline
0&I_{4g-5}\\
 \end{array}
\right).\nonumber
\end{equation}
Considering different $s\in S'_{4g-4}$ acting on $s_{1}h_{e_{i}}s^{-1}_{1}s_{2}h_{e_{i}}s^{-1}_{2}$, one can obtain the matrices $\{A_{e_{i}}^{k}\}_{k=2}^{4g-5}$ with $R^{k}$ as the only non-zero row corresponding to $\Delta_{e_{i}}$.
Thus, by composing the elements of $N$ to which each $A^{k}_{e_{i}}$ corresponds, we can obtain any possible distribution of an even number of $1$'s in the only non-zero row.

Similar arguments  can be used for the matrices corresponding to $u_{5},l_{6}$ which in this case may have any number of $1$'s in the only non-zero row. 
From Remark \ref{remi2} we are then able to generate any matrix $A\in N$ as described in the proposition.\end{proof}

From (\ref{semid}), recall that $G_{1}$ is a semi-direct product of $S'_{4g-4}$ and $N$. Thus, from the description of the action of the semi direct product  in Theorem \ref{teomuy}, and the study of the kernel $N$ of Proposition \ref{muy},   we have the following theorem:

\begin{theorem}\label{arribat}
 The representation of any $\sigma \in G_{1}$ in the basis $\beta$ can be written as
\begin{equation}
 [\sigma]=\left(\begin{array} {c|c}
I&A\\\hline
0&\pi                 
                \end{array}
\right) ~,
\end{equation}
where  $\pi$ represents a permutation on $\mathbb{Z}_{2}^{4g-5}$ and  $A$ is a matrix whose rows corresponding to $\Delta_{e}$ for $e\in E_{0}-\{u_{5},l_{6}\}$ have an even number of $1$'s, the row corresponding to $x_{4}$ is zero and the rows corresponding to $\Delta_{u_{5}}$ and $\Delta_{l_{6}}$ have any distribution of $1$'s.
\end{theorem}

\begin{remark}
 Note that Theorem \ref{arribat} implies that as $\sigma$ runs through all elements in $G_{1}$, all possible matrices $A$ as described in the theorem appear.
\end{remark}

\section{ The monodromy action of $\pi_{1}(\A)$ on $P[2]$} \label{sec:group3}

As seen previously,  $P[2]$ can be obtained as the quotient of $C_{1}$ by the four relations $x_{1},x_{2},x_{4}$ and $x_{5}$.

\begin{proposition}
 The relations $x_{1},x_{2},x_{4}$ and $x_{5}$  are preserved by the action of $\pi_{1}(\A)$. 
\end{proposition}
\begin{proof}
 Any edge $e\in E$  is adjacent to either 2 or no edge in $x_{1}$  (or $x_{2},x_{5}$) thus labelling a trivial action on $x_{1}, x_{2}$ and $x_{5}$. Similarly, any edge $e\in E$ is adjacent to 2 or 4 edges in $x_{4}$ and thus is also trivial on $x_{4}$.
\end{proof}
 Note that $x_{1},x_{2},x_{4} \in Z_{1}$ and  $\partial  x_{5} = (1,1,\cdots,1)\in B_{0}$. Hence, we have the following maps
\begin{equation}\mathbb{Z}_{2}^{2g}\cong \frac{Z_{1}}{<x_{1},x_{2},x_{4}>}\rightarrow P[2]\rightarrow \frac{B_{0}}{<(1,1,\cdots,1)>}\cong \mathbb{Z}_{2}^{4g-6}~. \label{exact}\end{equation}
\begin{definition}The monodromy group $G$ is given by the action  on the quotient $P[2]$ induced by the action of $G_{1}$ on $C_{1}$. 
\end{definition}
Note that $x_{1},x_{2},x_{4}=0$ imply 
\begin{eqnarray}
 \Delta_{l_{6}}=\sum_{l_{i}\in E_{0}-\{l_{6}\}}\Delta_{l_{i}}\label{sum1}~;~{~}~
\Delta_{u_{5}}=\sum_{l_{i}\in E_{0}-\{u_{5}\}}\Delta_{u_{i}}~.\label{sum2}
\end{eqnarray}
Moreover, as $x_{1}+x_{2}+x_{4}+x_{5}=0$, one can express the edge $u_{6}$ in terms of  elements in $\beta'-\{u_{6}\}$.

\begin{definition}
  For $\tilde{\beta}_{0}:=\beta_{0}-\{x_{4},\Delta_{l_{6}},\Delta_{u_{5}} \}$ and $\tilde{\beta}':=\beta'-\{u_{6}\}$, let $\tilde{\beta}:=\{\tilde{\beta}_{0},\tilde{\beta}'\}$. 
\end{definition}

The elements in $\tilde{\beta}_{0}$ generate 
$Z_{1}/<x_{1},x_{2},x_{4}>$,
and $\tilde{\beta}'$ generates a complementary subspace in $P[2]$. 
Hence, we obtain our main theorem, which gives an explicit description of the monodromy action on $P[2]$:

\setcounter{theorem}{0}

\begin{theorem}\label{teo}
 The monodromy group $G\subset GL(6g-6,\mathbb{Z}_{2})$ on the first mod 2 cohomology of the fibres of the Hitchin fibration is given by the group of matrices   of the form
\begin{equation}
 [\sigma]=\left(\begin{array} {c|c}
I_{2g}&A\\\hline
0&\pi                 
                \end{array}
\right),
\end{equation}
where
\begin{itemize}
                   \item  $\pi$ is the quotient action on $\mathbb{Z}_{2}^{4g-5}/(1,\cdots,1)$ induced by the permutation representation on $\mathbb{Z}_{2}^{4g-5}$,
\item $A$  is any $(2g)\times (4g-6)$ matrix with entries in $\mathbb{Z}_{2}$.
                  \end{itemize}
\end{theorem}

\begin{proof} From the above analysis,  the action of $G$ on $B_{0}/<\partial x_{5}=(1,1,\cdots,1)>$ is given by the quotient action of the symmetric group. Furthermore, replacing $\Delta_{u_{5}}$ and $\Delta_{l_{6}}$ by the sums in (\ref{sum1}), one can use similar arguments to the ones in Proposition \ref{muy}  to obtain any number of 1's in all the rows of the matrix $A$.
\end{proof}
\setcounter{theorem}{18}


\begin{remark}We have seen in Proposition \ref{trivial} that the action of $G_{1}$ is trivial on $Z_{1}$. Moreover, the space $\mathbb{Z}_{2}[x_{1},x_{2},x_{4},x_{5}]$ is preserved by the action of $G_{1}$, and thus one can see that the induced monodromy action  on the $2g$-dimensional subspace  $Z_{1}/\mathbb{Z}_{2}[x_{1},x_{2},x_{4},x_{5}]$ is trivial. Geometrically this space is defined as follows. There are $2^{2g}$ sections of the Hitchin fibration given by the choice of the square root of $K$, which identifies the fibre with the Prym variety.  These sections meet each Prym in $2^{2g}$ points which also lie in $P[2]$. Since we can lift a closed curve by  a section, these are acted on trivially by the monodromy.
\end{remark}

\section{The orbits of the monodromy group}\label{sec:orbit}

We shall study now the orbits $G_{(s,x)}$  of each $(s,x)\in P[2]\cong \mathbb{Z}_{2}^{2g}\oplus\mathbb{Z}_{2}^{4g-6} $ under the action of $G$. Note that for $g\in G$, its action  on $(s,x)$ is given by

\begin{equation}
 g\cdot (s,x)^{t}=\left(\begin{array}
  {c|c}I&A\\ \hline 0&\pi
 \end{array}\right)\left(\begin{array}{c}s\\x
\end{array}\right)=\left(\begin{array}{c}s+Ax\\\pi x
\end{array}\right) ~.\label{arriba}
\end{equation}

\begin{proposition}\label{numero}
 The action of $G$ on $P[2]$ has $2^{2g}+g-1$ different orbits.
\end{proposition}

\begin{proof} The matrices $A$ have any possible number of $1$'s in each row,  and so for $x \neq 0$  any $s'\in \mathbb{Z}_{2}^{2g}$ may be written as  $s'=s+Ax$ for some $A$. Hence, the number of orbits $G_{(s,x)}$ for $x\neq0$ is determined by the number of orbits of the action in $\mathbb{Z}_{2}^{4g-6}$ defined by 
$\xi:x\rightarrow \pi x$.  
Since the map $\xi$ permutes the non-zero entries of $x$,   the orbits of this action are given by elements with the same number of $1$'s.

From equation (\ref{exact}) the space $\mathbb{Z}_{2}^{4g-6}$ can be thought of as  vectors in $ \mathbb{Z}_{2}^{4g-4}$  with an even number of $1$'s modulo $(1,\dots,1)$. 
  Thus, for $x\in \mathbb{Z}_{2}^{4g-6}$ and $x\neq 0$, each orbit $\xi_{x}$ is defined by a constant $m$ such that $x$ has $2m$ non-zero entries, for $0<2m\leq4g-4$. We shall denote by $\bar{x}\in \mathbb{Z}^{4g-6}$  the  element  defined by the constant $\bar{m}$ such that $2\bar{m}=(4g-4)-2m$. With this notation, we  can see that $\bar{x}$ and $x$ belong to the same equivalence class in $\mathbb{Z}_{2}^{4g-6}$.  Note that for $m\neq g-1,2g-2$, the corresponding $x$ is equivalent to $\bar{x}$ under $\xi$ and thus in this case there are $g-2$ equivalent classes $\xi_{x}$. Then, considering  the equivalence class for $m=g-1$, one has $g-1$ different classes for the action of $\xi$.
From equation (\ref{arriba}), the action of $G$ on an element $(s,0)\in P[2]$ is trivial, and so one has $2^{2g}$ different orbits of $G$.
\end{proof}

\section{Connected components of $\mathcal{M}_{SL(2,\mathbb{R})}$} \label{sec:component}
As seen in Section \ref{sl}, the fixed point set in $\mathcal{M}$ of the involution $\Theta$ gives the subspace  of semistable $SL(2,\mathbb{R})$-Higgs bundles.  As an application of our results, we shall study the connected components corresponding to stable and strictly semistable $SL(2,\mathbb{R})$-Higgs bundles. In particular, we shall check  that no connected component of the fixed point set of $\Theta$ lies entirely over the discriminant locus of $\mathcal{A}$. 

Recall that the regular fibres of the  Hitchin fibration $h: \M\rightarrow \A$ are isomorphic to ${\rm Prym}(S,\Sigma)$, and that $SL(2,\mathbb{R})$-Higgs bundles correspond to the points of order 2 in these fibres, which we have denoted by $P[2]$. Hence, there is a finite covering of $\A$ of degree $2^{6g-6}$ whose fibres are isomorphic to  $P[2]$. 
In this section we  first  calculate the number of connected components of $\mathcal{M}_{SL(2,\mathbb{R})}$ via the inclusion $\mathcal{M}_{SL(2,\mathbb{R})}\subset \mathcal{M}$, and then add the number of connected components  which could not be seen through the monodromy analysis.

If $p$ and $q$ are in the same orbit of the monodromy group, they are connected by the lift of a closed path in $\A$. Conversely if $p$ and $q$ are connected in $\mathcal{M}_{SL(2,\mathbb{R})}$, the monodromy around the projection of the path takes one to the other.  Hence, if $p$ and $q$ are in the same orbit, then they are in the same connected component  of the fixed point submanifold of  $\Theta:(V,\Phi) \mapsto (V,-\Phi)$. 

In the case of strictly semistable Higgs bundles, the following Proposition applies:

\begin{proposition}\label{part1}
 Any point representing a strictly semistable Higgs bundle in $\mathcal{M}$  fixed by $\Theta$ is in the connected component of a Higgs bundle with zero Higgs field.
\end{proposition}

\begin{proof} The moduli space $\mathcal{M}$ is the space of $S$-equivalence classes of semistable Higgs bundles. A strictly semistable $SL(2,\mathbb{C})$-Higgs bundle $(E,\Phi)$ is represented by  $E=V\oplus V^{*}$ for a degree zero line bundle $V$, and
\[\Phi=\left(\begin{array}
              {cc}a&0\\0&-a
             \end{array}\right)~{~\rm for~}~a\in H^{0}(\Sigma,  K).\]
 
If $V^2$ is nontrivial, only the automorphisms fix $\Phi$ so this is   an $SL(2,\mathbb{R})$-Higgs bundle only for $\Phi= 0$, and corresponds to a flat connection with  holonomy 
in $SO(2)\subset SL(2,R)$. If $V^2$ is trivial then the automorphism $(u,v)\mapsto (v,-u)$ takes $\Phi$ to $-\Phi$, corresponding to a flat bundle with holonomy in 
$\mathbb{R}^*\subset SL(2,\mathbb{R})$.     By scaling $\Phi$ to zero this is connected to the zero  Higgs field.
The differential $a$ can be continuously deformed to zero by considering $ta$ for $0\leq t\leq 1$.  Hence, by stability of line bundles,   one can continuously deform  $(E,\Phi)$ to a Higgs bundle with zero Higgs field via strictly semistable pairs.  
\end{proof}

In the case of stable $SL(2,\mathbb{R})$-Higgs bundles we have the following result:

\begin{proposition}\label{prop2}  Any stable $SL(2,\mathbb{R})$-Higgs bundle is in a connected component which intersects $\M$.
\end{proposition}

\begin{proof} Let $(E=V\oplus V^{*},\Phi)$ be a stable $SL(2,\mathbb{R})$-Higgs bundle with  
\begin{equation}
 \Phi=\left(\begin{array}
              {cc}0&\beta\\\gamma&0
             \end{array}\right)\in H^{0}(\Sigma,{\rm End}_{0}E\otimes K),\nonumber
\end{equation}
and let $d:={\rm deg}(V)\geq 0$. Stability implies that   the section $\gamma\in H^{0}(\Sigma, V^{-2}K)$ is non-zero, and thus   $0\leq 2d \leq 2g-2$. Moreover, the section $\beta$ of $V^{2}K$ can be deformed continuously to zero without affecting  stability.

 The section $\gamma$ defines a point $[\gamma]$ in the symmetric product $S^{2g-2-2d}\Sigma$. As this space is connected, one can continuously deform the divisor $[\gamma]$ to any $[\tilde{\gamma}]$ composed of distinct points. For   $a\in \A$ with zeros $x_{1},\ldots,x_{4g-4} $, we may deform $[\gamma]$ to the divisor  $[\tilde{\gamma}]$ given by the points $x_{1},\ldots,x_{n}\in \Sigma$ for  $n:= 2g-2-2d$, and such that $\tilde{\gamma}$ is a section of $U^{-2}K$ for some line bundle $U$.

 The complementary zeros  $x_{n+1},\ldots, x_{g-4}$ of $a$ form a divisor of $U^{2}K$. Any section $\tilde{\beta}$ with this divisor can be reached by continuously deforming $\beta$ from zero to $\tilde{\beta}$ through sections without affecting  stability. Hence,  we may continuously deform any stable pair  $(V\oplus V^{*},\Phi=\{\beta,\gamma\})$ to a Higgs bundle $(U\oplus U^{*}, \tilde{\Phi}=\{\tilde{\beta}, \tilde{\gamma}\})$ in the regular fibres over $\A$.
\end{proof}

The above analysis establishes  that   the number of connected components of the fixed point set of the involution $\Theta$ on $\mathcal{M}$ is less than or equal to the number of orbits of the monodromy group in $P[2]$.
 A flat $SL(2,\mathbb{R})$-Higgs bundle has an associated $\mathbb{R}\mathbb{P}^{1}$ bundle whose Euler class $k$ is a topological invariant which satisfies the Milnor-Wood  inequality $-(g-1)\leq k\leq g-1$.  In particular,  $SL(2,\mathbb{R})$-Higgs bundles with different Euler class lie in different connected components $\mathcal{M}^{k}_{SL(2,\mathbb{R})}$ of the fixed  point set of $\Theta$. Moreover, for $k=g-1, -(g-1)$ one has $2^{2g}$ connected components corresponding to the so-called Hitchin components. Hence, the lower bound  to the number of connected components of the fixed point set of the involution $\Theta$ for $k\geq 0$ is $2^{2g}+g-1$. As this lower bound equals the number of orbits of the monodromy group on the fixed points of $\Theta$ on $\M$, one has that the closures of these orbits can not intersect.   

Since we consider $k\geq 0$, the number of connected components of the fixed points of the involution $\Theta:~(E,\Phi)\mapsto (E,-\Phi)$ on  $\mathcal{M}_{reg}$ is equal to the number of orbits of the monodromy group on the points of order two of the regular fibres $\M$, i.e. $2^{2g}+g-1$. From the above analysis, one has the following application of our main theorem:

\begin{corollary} The number of connected components of the moduli space of semistable $SL(2,\mathbb{R})$-Higgs bundles  as $SL(2,\mathbb{C})$-Higgs bundles is $2^{2g}+g-1$. 
\end{corollary}

\setcounter{theorem}{1}
As mentioned previously,  the Euler classes $-k$ for $0<k< g-1$  should also be considered, since this invariant labels connected components  $\mathcal{M}^{-k}_{SL(2,\mathbb{R})}$
which are mapped into $\mathcal{M}^{k}_{SL(2,\mathbb{R})}\subset \mathcal{M}.$
Hence, we have a decomposition 
\begin{eqnarray}\mathcal{M}_{ SL(2,\mathbb{R})}= \mathcal{M}^{k=0}_{ SL(2,\mathbb{R})} \sqcup \left(\bigsqcup_{i=1}^{2^{2g}}  \mathcal{M}^{k=\pm (g-1),i}_{SL(2,\mathbb{R})}    \right) \sqcup  \left( \bigsqcup_{k=1}^{g-2} \mathcal{M}^{\pm k}_{ SL(2,\mathbb{R})} \right),\nonumber \end{eqnarray}
which implies the following result:
\begin{proposition} The number of connected components of the moduli space of semistable $SL(2,\mathbb{R})$-Higgs bundles  is $
2\cdot 2^{2g}+2g-3
$. \label{coro}
\end{proposition}

\setcounter{theorem}{23}

The construction of the orbits of the monodromy group provides a decomposition of the  $4g-4$ zeros of $\det \Phi$ via the $2m$ non-zero entries in Proposition \ref{numero}.
Considering the notation of Section \ref{sec2}, the distinguished subset of zeros correspond to the points in the spectral curve $S$ where the action on the line bundle $L$ is trivial.

\section*{Acknowledgements} The author is thankful to Prof. Nigel Hitchin for suggesting this problem and for  many helpful discussions, and to S. Ramanan for pointing out a subtlety that needed attention. This work was funded by the Oxford University Press through a Clarendon Award, and by New College, Oxford.

\end{document}